\documentclass[12pt,a4paper,reqno]{amsart}

\usepackage[T1]{fontenc}
\usepackage[utf8]{inputenc}
\usepackage[english]{babel}

\usepackage[%
	left=2.5cm,       
	right=2.5cm,      
	top=3.5cm,        
	bottom=3.5cm,     
	heightrounded,
	bindingoffset=0mm 
]{geometry}

\usepackage{amsmath}
\usepackage{amsfonts}
\usepackage{amssymb}
\usepackage{lmodern}
\usepackage{mathrsfs}
\usepackage{mathtools}

\numberwithin{equation}{section}

\usepackage{hyperref}
\usepackage{xcolor}

\usepackage{bookmark}

\usepackage[initials,
]{amsrefs}
\renewcommand{\MR}[1]{} 

\usepackage[noabbrev,capitalize]{cleveref}

\usepackage{enumitem}

\theoremstyle{definition}
\newtheorem{Def}{Definition}[section]

\theoremstyle{plain}
\newtheorem{prop}[Def]{Proposition}

\newtheorem{thm}[Def]{Theorem}
\newtheorem{lem}[Def]{Lemma}

\theoremstyle{definition}

\DeclareMathOperator{\di}{d\!}
\DeclareMathOperator{\esssup}{ess\,sup}
\DeclarePairedDelimiter{\abs}{|}{|}
\DeclarePairedDelimiter{\norm}{\|}{\|}
\DeclarePairedDelimiter{\set}{\{}{\}}

\renewcommand{\phi}{\varphi}
\renewcommand{\rho}{\varrho}

\newcommand{\e}{\varepsilon}

\newcommand{\divergence}[1]{\text{div}\left({#1}\right)}
\newcommand{\R}{\mathbb{R}}

\newcommand{\N}{\mathbb{N}}

\newcommand{\Lip}{\text{Lip}} 
\newcommand{\MM}{\mathcal{M}}
\newcommand{\ws}{\!-\!w^*}

\allowdisplaybreaks
\makeatletter
\newcommand*{\mint}[1]{%
  \mint@l{#1}{}%
}
\newcommand*{\mint@l}[2]{%
  \@ifnextchar\limits{%
    \mint@l{#1}%
  }{%
    \@ifnextchar\nolimits{%
      \mint@l{#1}%
    }{%
      \@ifnextchar\displaylimits{%
        \mint@l{#1}%
      }{%
        \mint@s{#2}{#1}%
      }%
    }%
  }%
}
\newcommand*{\mint@s}[2]{%
  \@ifnextchar_{%
    \mint@sub{#1}{#2}%
  }{%
    \@ifnextchar^{%
      \mint@sup{#1}{#2}%
    }{%
      \mint@{#1}{#2}{}{}%
    }%
  }%
}
\def\mint@sub#1#2_#3{%
  \@ifnextchar^{%
    \mint@sub@sup{#1}{#2}{#3}%
  }{%
    \mint@{#1}{#2}{#3}{}%
  }%
}
\def\mint@sup#1#2^#3{%
  \@ifnextchar_{%
    \mint@sup@sub{#1}{#2}{#3}%
  }{%
    \mint@{#1}{#2}{}{#3}%
  }%
}
\def\mint@sub@sup#1#2#3^#4{%
  \mint@{#1}{#2}{#3}{#4}%
}
\def\mint@sup@sub#1#2#3_#4{%
  \mint@{#1}{#2}{#4}{#3}%
}
\newcommand*{\mint@}[4]{%
  \mathop{}%
  \mkern-\thinmuskip
  \mathchoice{%
    \mint@@{#1}{#2}{#3}{#4}%
        \displaystyle\textstyle\scriptstyle
  }{%
    \mint@@{#1}{#2}{#3}{#4}%
        \textstyle\scriptstyle\scriptstyle
  }{%
    \mint@@{#1}{#2}{#3}{#4}%
        \scriptstyle\scriptscriptstyle\scriptscriptstyle
  }{%
    \mint@@{#1}{#2}{#3}{#4}%
        \scriptscriptstyle\scriptscriptstyle\scriptscriptstyle
  }%
  \mkern-\thinmuskip
  \int#1%
  \ifx\\#3\\\else_{#3}\fi
  \ifx\\#4\\\else^{#4}\fi  
}
\newcommand*{\mint@@}[7]{%
  \begingroup
    \sbox0{$#5\int\m@th$}%
    \sbox2{$#5\int_{}\m@th$}%
    \dimen2=\wd0 %
    \let\mint@limits=#1\relax
    \ifx\mint@limits\relax
      \sbox4{$#5\int_{\kern1sp}^{\kern1sp}\m@th$}%
      \ifdim\wd4>\wd2 %
        \let\mint@limits=\nolimits
      \else
        \let\mint@limits=\limits
      \fi
    \fi
    \ifx\mint@limits\displaylimits
      \ifx#5\displaystyle
        \let\mint@limits=\limits
      \fi
    \fi
    \ifx\mint@limits\limits
      \sbox0{$#7#3\m@th$}%
      \sbox2{$#7#4\m@th$}%
      \ifdim\wd0>\dimen2 %
        \dimen2=\wd0 %
      \fi
      \ifdim\wd2>\dimen2 %
        \dimen2=\wd2 %
      \fi
    \fi
    \rlap{%
      $#5%
        \vcenter{%
          \hbox to\dimen2{%
            \hss
            $#6{#2}\m@th$%
            \hss
          }%
        }%
      $%
    }%
  \endgroup
}

\begin{document}

\title[Lagrangian stability for a non-local CE system under Osgood condition]{Lagrangian stability for a system of non-local continuity equations under  Osgood condition}

\author[M.~Inversi]{Marco Inversi}
\address[M.~Inversi]{Department Mathematik und Informatik, Universität Basel, Spiegelgasse~1, 4051 Basel, Switzerland}
\email{marco.inversi@unibas.ch}

\author[G.~Stefani]{Giorgio Stefani}
\address[G.~Stefani]{Scuola Internazionale Superiore di Studi Avanzati (SISSA), via Bonomea~265, 34136 Trieste (TS), Italy}
\email{gstefani@sissa.it {\normalfont \it or} giorgio.stefani.math@gmail.com}



\thanks{
\textit{Acknowledgements}.
The authors thank Gianluca Crippa for several useful comments on a preliminary version of this work.
The first-named author is partially funded by the SNF grant FLUTURA: Fluids, Turbulence, Advection No.\ 212573. The second-named author is member of the Istituto Nazionale di Alta Matematica (INdAM), Gruppo Nazionale per l'Analisi Matematica, la Probabilità e le loro Applicazioni (GNAMPA), is partially supported by the INdAM--GNAMPA 2022 Project \textit{Analisi geometrica in strutture subriemanniane}, codice CUP\_E55\-F22\-00\-02\-70\-001, and has received funding from the European Research Council (ERC) under the European Union’s Horizon 2020 research and innovation program (grant agreement No.~945655).
}

\begin{abstract}
We extend known existence and uniqueness results of weak measure solutions for systems of non-local continuity equations beyond the usual Lipschitz  regularity.
Existence of weak measure solutions holds for uniformly continuous vector fields and convolution kernels, while uniqueness follows from a Lagrangian stability estimate under an additional Osgood condition. 
\end{abstract}

\date{\today}

\maketitle

\section{Introduction}

\subsection{Statement of the problem}

For fixed $T\in(0,+\infty)$ and $k,d\in\N$, 
we consider the system of non-local continuity equations
\begin{equation} 
\label{non-local continuity}
\left\{
\begin{array}{rcll}
\partial_t \rho^i + \divergence{\rho^i\, V^i(t,x, \rho*\eta^i)} &=& 0, & t\in(0,T),\ x\in\R^d,
\\[2ex] 
\rho^i(0) &=& \bar{\rho}^i,   & i=1, \dots, k,
\end{array}
\right.
\end{equation}
where the unknown $\rho=(\rho^1, \dots, \rho^k)\in L^\infty([0,T]; \MM^+(\R^d)^k)$ is a time-dependent $k$-vector of non-negative Borel measures on~$\R^d$, the initial datum $\bar\rho=(\bar\rho^1,\dots,\bar\rho^k)\in\MM^+(\R^d)^k$ is a $k$-vector of non-negative Borel measures,
\begin{equation}
V=(V^1,\dots,V^k)\in L^\infty([0,T];C_b(\R^d \times \R^k;\R^d)^k)
\label{V def}
\end{equation} 
is a uniformly-in-time bounded continuous $k$-vector field
and 
\begin{equation}
\eta^i=(\eta^{i,1}, \dots, \eta^{i,k})\in L^\infty([0,T];C_b(\R^d;\R^k))
\label{eta def}
\end{equation}
is a uniformly-in-time bounded continuous $k$-vector of convolution kernels, with the convolution
$\rho*\eta^i = 
(\rho^1*\eta^{i,1}, \dots, \rho^k*\eta^{i,k})$
occurring in the space variable only.
In the entire paper, we frequently consider the $1$-norm (i.e., the sum of the absolute values of the entries) on both vectors and matrices.
In particular, $|\rho|=|\rho^1|+\dots+|\rho^k|$ and thus $\|\rho\|_{\MM}=\|\rho^1\|_{\MM}+\dots+\|\rho^k\|_{\MM}$ for all $\rho\in\MM(\R^d)$. 
When considering other norms, constants depending on $d$ and/or $k$ may be dropped without notice.

Solutions of the system~\eqref{non-local continuity} are understood in the usual distributional sense, which is well-set thanks to~\eqref{V def} and~\eqref{eta def}.  

\begin{Def}[Weak solution] 
\label{weak solution}
We say that $\rho \in L^\infty([0,T]; \MM^+(\R^d)^k)$ is a \emph{weak solution} of the system \eqref{non-local continuity} starting from the initial datum $\bar\rho\in \MM^+(\R^d)^k$ if
\begin{equation}
\int_0^T\int_{\R^d} \left(
\partial_t\phi
+
V^i(t,x,\rho*\eta^i)\cdot\nabla\phi
\right)\di\rho^i(t,x)\di t
+
\int_{\R^d}\phi(0,x)\di\bar\rho^i(x)=0
\label{weak formulation}
\end{equation} 
for each $i=1, \dots,k$ and any $\phi \in C^{\infty}_c([0,T)\times\R^d)$.
\end{Def}

Any weak solution in the sense of \cref{weak solution} admits a weakly continuous representative  in duality with the space $C_0(\R^d)$ of continuous functions vanishing at infinity, see~\cite{AGS08}*{Lem.~8.1.2} and~\cites{AB08,CL13}.
So, from now on, we restrict our attention to weakly-continuous weak solutions $\rho \in C([0,T];\MM^+(\R^d)^k\ws)$ only.

The system~\eqref{non-local continuity} is used in several physical situations---for instance, pedestrian traffic, sedimentation models and supply chains---to describe the time evolution of the density of a vectorial quantity (e.g., pedestrians or particles), possibly concentrating in some points or along hypersurfaces.
Far from being complete, we refer the reader for example  to~\cites{CL12-a,CL12-n,CHM11,R90,Z99,AMRKJ06,DF13,MKB14,CMR16,KP17} for a panoramic on the related literature.
Because of the physical relevance of the system~\eqref{non-local continuity}, here we deal with non-negative solutions only.

The system~\eqref{non-local continuity} can be also interpreted in the sense of the Control Theory.
Indeed, the convolution kernel $\eta$ can be viewed as a \emph{non-local} control for the (non-linear) PDE in~\eqref{non-local continuity}. 
Therefore, assuming $V$ is fixed for simplicity, any stability result for the solutions of the system \eqref{non-local continuity} in terms of the convolution kernel~$\eta$ yields a continuous dependence of the  curve $t\mapsto\rho_t[\eta]$ solving \eqref{non-local continuity} in terms of the control~$\eta$.

The well-posedness of the system~\eqref{non-local continuity} was established in~\cite{CL13}, provided that $V$ and $\eta$ are bounded and Lipschitz continuous uniformly in time, namely,
\begin{equation}
V\in L^\infty([0,T]; \Lip_b(\R^d\times \R^k; \R^d)^k)
\quad\text{and}\quad
\eta\in  L^{\infty}([0,T]; \Lip_b(\R^d;\R^k)^k).
\label{lip ass}
\end{equation}
The crucial ingredient of~\cite{CL13} is a stability estimate (in terms of the $1$-Was\-ser\-stein distance between two solutions, see~\cite{CL13}*{Prop.~4.2}) which, in turn, allows to obtain existence and uniqueness of the solution of~\eqref{non-local continuity} via a fix point argument. 
The idea of exploiting the Lipschitz regularity to gain stability of trajectories has been later applied to several other related problems, see~\cites{PR14,CJLV16,EHM16,BS20,CDKP22} and the references therein for instance.

\subsection{Main results}

The aim of the present note is to prove the well-posedness of the system~\eqref{non-local continuity} under less restrictive assumptions than~\eqref{lip ass}, that is, to extend the existence and uniqueness result of~\cite{CL13} beyond the Lipschitz regularity.
Our interest is motivated by some recent works~\cites{BN21,ANS22,CS21,L22,AB08,LL15} dealing with non-Lipschitz velocity fields.

Our first main result deals with the existence of weak solutions of the system~\eqref{non-local continuity}, in the spirit of the celebrated Peano's Theorem. 
To this aim, we consider the following structural hypotheses (where \emph{modulus of continuity} means a non-decreasing concave function  vanishing continuously at zero):

\begin{enumerate}[label=($V$),ref=$V$]

\item\label{ass V}
The vector field $V\in L^\infty([0,T];C_b(\R^d \times \R^k;\R^d)^k)$ satisfies
\begin{equation}
\underset{t\in[0,T]}{\esssup}\,
\abs{V(t,x,u) - V(t,y,v)} \leq \omega_V(\abs{x-y} + \abs{u-v})
\quad 
\forall x,y \in \R^d,\ u,v \in \R^k,
\label{omega V}
\end{equation}
where $\omega_V\colon[0,+\infty)\to[0,+\infty)$ is a modulus of continuity.

\end{enumerate}

\begin{enumerate}[label=($\eta$),ref=$\eta$]

\item\label{ass eta}
For each $i=1,\dots,k$, the convolution kernel $\eta^i\in L^\infty([0,T];C_0(\R^d;\R^k))$
satisfies
\begin{equation}
\underset{t\in[0,T]}{\esssup}\,
\abs{\eta^i(t,x)-\eta^i(t,y)} \leq \omega_\eta(\abs{x-y})
\quad
\forall x,y\in\R^d,
\label{omega eta}
\end{equation} 
where $\omega_\eta\colon[0,+\infty)\to[0,+\infty)$ is a modulus of continuity.

\end{enumerate}

\begin{thm}[Existence] \label{existence}
If \eqref{ass V} and \eqref{ass eta} hold, then the system~\eqref{non-local continuity} admits a weak solution starting from any given initial datum in $\MM^+(\R^d)^k$. 
\end{thm}

To prove \cref{existence}, we first consider the smoothed functions $V_\e$ and $\eta_\e$ and obtain a weak solution $\rho_\e$ of the corresponding system~\eqref{non-local continuity} for all $\e>0$ in virtue of the main result of~\cite{CL13}. Then, we pass to the limit as $\e\to0^+$ showing that $\rho_\e$ (weakly) converges to a weak solution of the system~\eqref{non-local continuity}. 
The needed a priori compactness is achieved via an Aubin--Lion-type Lemma which is inspired by~\cite{CS21}*{Th.~A.1}.  

In order to achieve uniqueness of weak solutions of the system~\eqref{non-local continuity}, we need to impose a further  \emph{Osgood condition} on the composition of the two moduli of continuity of $V$ and~$\eta$:

\begin{enumerate}[label=($O$),ref=$O$]

\item 
\label{ass osgood}
for each $\lambda>0$, it holds
\begin{equation*}
\int_{0+}\frac{\di r}{\omega_V(r+\lambda\,\omega_\eta(r))}=+\infty.
\end{equation*}

\end{enumerate}

For example, condition~\eqref{ass osgood} is satisfied as soon as $\omega_V\circ\omega_\eta$ is a $\log$-linear function, such as $r|\log r|$, $r\log|\log r|$ and similar, with $r>0$ sufficiently small.

Our uniqueness result deals with \emph{Lagrangian} weak solutions of the system~\eqref{non-local continuity}.

\begin{Def}[Lagrangian weak solution]
A weak solution $\rho \in L^\infty([0,T]; \MM^+(\R^d)^k)$ of the system~\eqref{non-local continuity} starting from the initial datum $\bar\rho\in \MM^+(\R^d)^k$ is \emph{Lagrangian} if 
\begin{equation*}
\rho^i(t,\cdot)=X^i(t,\cdot)_\#\bar\rho^i,
\quad
i=1,\dots,k,
\end{equation*}
where $X^i\colon[0,T]\times\R^d\to\R^d$ is the (classical) solution of the ODE 
\begin{equation}
\label{ode}
\left\{
\begin{array}{rcll}
\dfrac{\di}{\di t}\,X^i(t,x) 
&=&
V^i\big(t,X^i(t,x),
\rho*\eta^i(t,X^i(t,x))\big),
\quad
& t\in(0,T),\ x\in\R^d,
\\[2ex]
X^i(0,x) &=& x,
\quad
& x\in\R^d.
\end{array}
\right.
\end{equation}
\end{Def} 

Thanks to \cref{vector field omega} below, the Osgood condition in~\eqref{ass osgood} ensures the well-po\-sed\-ness of the ODE in~\eqref{ode}.

\begin{prop}[Associated vector field]
\label{vector field omega}
Let assumptions \eqref{ass V} and \eqref{ass eta} be in force.
If $\rho\in C([0,T]; \MM^+(\R^d)^k\ws)$ is a weak solution of the system~\eqref{non-local continuity} starting from the initial datum $\bar\rho\in \MM^+(\R^d)^k$, then the vector field 
\begin{equation}
b_{V,\eta,\rho}^i(t,x)=V^i\big(t,x,\rho*\eta^i(t,x)\big),
\quad
t\in[0,T],\ x\in\R^d,\
i=1,\dots,k,
\end{equation}
appearing in~\eqref{ode}
satisfies $b\in L^\infty([0,T];C_b(\R^d;\R^d)^k)$ with 
\begin{equation*}
\underset{t\in[0,T]}{\esssup}\,
|b_{V,\eta,\rho}(t,x)-b_{V,\eta,\rho}(t,y)|
\lesssim 
\omega_V\big(|x-y|+\|\bar\rho\|_{\MM}\,\omega_\eta(|x-y|)\big)
\quad
\forall
x,y\in\R^d.
\end{equation*}
\end{prop} 

With the above notation, our main uniqueness result reads as follows.

\begin{thm}[Uniqueness]
\label{uniqueness}
If \eqref{ass V}, \eqref{ass eta} and~\eqref{ass osgood} hold, then~\eqref{non-local continuity} admits a unique Lagrangian weak solution starting from any given initial datum in $\MM^+(\R^d)^k$. 
\end{thm}

The word ``Lagrangian'' in \cref{uniqueness} can be dropped, since any weak solution of the system~\eqref{non-local continuity} is in fact Lagrangian because of~\cite{AB08}*{Th.~1} (also see~\cites{CJMO19}) and of \cref{vector field omega}. 
However, this regularity result is not at all elementary, so we prefer to state \cref{uniqueness} for Lagrangian solutions only in order to emphasize what is possible to achieve just relying on our elementary approach.


The strategy of~\cite{CL13} exploits the linearity of $\omega_\eta$ in an essential way. Indeed, the authors need the Lipschitz continuity of $\eta$ in order to recover the $1$-Wasserstein distance between two weak solutions of~\eqref{non-local continuity} in terms of its dual Kan\-to\-ro\-vich--Ru\-bin\-stein formulation (see~\cite{CL13}*{Lem.~4.1}).
We do not know if the strategy of~\cite{CL13} can be adapted to deal with a more general modulus of continuity~$\omega_\eta$.

To overcome this issue, we adopt a different point of view, which is inspired by the elementary uniqueness result achieved in the recent work~\cite{CS21}. 
Instead of controlling the $1$-Wasserstein distance between two weak solutions of the system~\eqref{non-local continuity}, we exploit their Lagrangian property to quantitatively estimate the difference between the two associated ODE flows, thus providing a \emph{Lagrangian} stability of weak solutions from which \cref{uniqueness} immediately follows.  

\begin{thm}[Lagrangian stability] \label{stability}
Let $V,U\in L^\infty([0,T];C_b(\R^d \times \R^k;\R^d)^k)$ satisfy~\eqref{omega V} with the same modulus of continuity $\omega_V$ and let $\eta,\nu\in L^\infty([0,T];C_0(\R^d;\R^k)^k)$ satisfy~\eqref{omega eta} with the same same modulus of continuity $\omega_\eta$. 
Let $\rho,\sigma\in C([0,T];\MM^+(\R^d)^k\ws)$ be two weak solutions of the system~\eqref{non-local continuity} starting from the initial data $\bar{\rho}, \bar{\sigma} \in\MM^+(\R^d)^k$, with vector fields $V,U$ and convolution kernels $\eta,\nu$, respectively.
Assume that $\rho,\sigma$ are Lagrangian, namely, 
$\rho=X(t,\cdot)_\#\bar\rho$
and 
$\sigma=Y(t,\cdot)_\#\bar\sigma$
for $t\in[0,T]$,
where $X,Y$ are the flows solving the corresponding ODEs in~\eqref{ode}.
Then, there exists a modulus of continuity $\Omega\colon[0,+\infty)\to[0,+\infty)$, only depending on
\begin{equation*}
T,\
\|\bar\rho\|_{\MM},\
\|\bar\sigma\|_{\MM},\
\|\eta\|_{L^\infty(C)},\
\|\nu\|_{L^\infty(C)},\
\omega_V,\
\omega_\eta,
\end{equation*}
such that 
\begin{equation}
\sup_{t\in[0,T]}
\|X(t,\cdot)-Y(t,\cdot)\|_{L^\infty}
\le 
\Omega\left(
\|\bar\rho-\bar\sigma\|_{\MM}
+
\|V-U\|_{L^\infty(C)}
+
\|\nu-\eta\|_{L^\infty(C)}
\right).
\label{stability estimate}
\end{equation}
\end{thm}

The modulus of continuity $\Omega$ in \cref{stability} can be explicitly computed as soon as one can invert the integral function
\begin{equation}
G_{V,\eta,\lambda}(r)
=
\int_{r_0}^r\frac{\di s}{\omega_V(s+\lambda\,\omega_\eta(s))},
\quad
r\ge0,\ r_0>0,
\label{bls}
\end{equation} 
naturally brought by the Osgood condition assumed in~\eqref{ass osgood}.
In fact, the stability estimate~\eqref{stability estimate} follows by simply differentiating a localized integral distance between the flows with respect to the time variable, and then applying the classical Bihari--LaSalle inequality (see \cite{pachpatte}*{Th.~2.3.1} for instance) with Osgood modulus of continuity $r\to\omega_V(r+\lambda\,\omega_\eta(r))$, for some specific parameter $\lambda>0$ depending on  $\|\bar\rho\|_{\MM}$ and $\|\bar\sigma\|_{\MM}$.

\cref{stability} clearly rephrases as a stability result of the flow of the ODE in~\eqref{ode}. 
From the point of view of Control Theory, the stability estimate in~\eqref{stability estimate} yields a continuous dependence of the (Lagrangian) solutions of the system \eqref{non-local continuity}, 
i.e., of the  flows induced by the corresponding ODE in~\eqref{ode}, 
in terms of the (non-local) control given by the convolution kernel, as well as of the velocity vector field and of the initial datum. 

\section{Proofs}

\subsection{Existence of weak solutions}

To prove \cref{existence}, we need some preliminary results.
We begin with an Aubin--Lions-type Lemma, which is inspired by~\cite{CS21}*{Th.~A.1}. 

\begin{lem}[Compactness] 
\label{Aubin-Lions}
Let $(\rho_n)_{n\in\N}\subset C([0,T];\MM(\R^d)\ws)$ be such that
\begin{equation}
\sup_{n\in\N}\,\norm{\rho_n}_{L^\infty(\MM)} < +\infty. \label{equibound norm}
\end{equation}
Assume that, for each $\phi \in C^\infty_c(\R^d)$, the functions $F_n[\phi]\colon[0,T]\to\R$, given by
\begin{equation*}
F_n[\phi](t)
=
\int_{\R^d}\phi\di\rho_n(t,\cdot),
\quad
t\in[0,T],
\end{equation*}
are uniformly equicontinuous on $[0,T]$, that is,
\begin{equation}
\forall\e>0\ \exists\delta>0\
:\
s,t\in[0,T],\ 
|s-t|<\delta\implies
\sup_{n\in\N}|F_n[\phi](s)-F_n[\phi](t)|<\e.
\label{equicont}
\end{equation}
Then, there exist a subsequence $(\rho_{n_k})_{k\in\N}$ and $\rho \in C([0,T], \MM(\R^d)\ws)$ such that 
\begin{equation}
\lim_{k\to+\infty}
\sup_{t\in[0,T]}\,
\abs*{\,\int_{\R^d}\phi\di\rho_{n_k}(t,\cdot)-\int_{\R^d}\phi\di\rho(t,\cdot)\,}=0
\label{al limit}
\end{equation}
for all $\phi\in C_0(\R^d)$.
\end{lem}

\begin{proof}
Let $\mathscr{D} \subset C_c(\R^d)$ be a countable and dense set in $C_0(\R^d)$. 
In virtue of~\eqref{equibound norm} and~\eqref{equicont}, for each $\phi \in \mathscr{D}$ the sequence $(F_n[\phi])_{n \in \N}$ is equibounded and equicontinuous on $[0,T]$. By Ascoli--Arzelà Theorem and a standard diagonal argument, we can  find a subsequence $({n_k})_{k\in\N}$ such that, for each $\phi\in\mathscr D$, the sequence $(F_{n_k}[\phi])_{k\in\N}$ is uniformly convergent to some $F[\phi]\in C([0,T])$, with
\begin{equation}
\norm{F[\phi]}_{L^\infty([0,T])}\le\norm{\phi}_{L^\infty}\,
\sup_{n\in\N}\,\norm{\rho_n}_{L^\infty(\MM)}.
\label{budino}
\end{equation}
By construction, the function $\phi\mapsto F[\phi](t)$ is a continuous linear functional on $\mathscr D$ for each $t\in[0,T]$.
Thus, for each fixed $t\in[0,T]$, we can extend the map $\phi\mapsto F[\phi](t)$ to a linear and continuous functional on $C_0(\R^d)$ for which we keep the same notation.
A plain approximation argument readily proves that, for each $\phi\in C_0(\R^d)$, the map $t\mapsto F[\phi](t)$ is continuous on $[0,T]$ and satisfies~\eqref{budino}.
By Riesz's Representation Theorem, for each $t\in[0,T]$ there exists a finite Borel measure $\rho(t,\cdot)\in\MM(\R^d)$ such that 
\begin{equation*}
F[\phi](t)=\int_{\R^d}\phi\di\rho(t,\cdot)
\quad
\text{for all}\ \phi\in C_0(\R^d),
\end{equation*}
so that $\rho\in C([0,T];\MM(\R^d)\ws)$.
Finally, in virtue of~\eqref{equibound norm} and~\eqref{budino}, for $\phi\in C_0(\R^d)$ and $\psi\in\mathscr D$, we can estimate
\begin{equation*}
\begin{split}
\sup_{t\in[0,T]}\abs{F_{n_k}[\phi](t)-F[\phi](t)}
&\le 
\sup_{t\in[0,T]}\abs{F_{n_k}[\psi](t)-F[\psi](t)}
+2\,\norm{\psi-\phi}_{L^\infty}\,\sup_{n\in\N}\,\norm{\rho_n}_{L^\infty(\MM)}
\end{split}
\end{equation*}
and the desired~\eqref{al limit} readily follows.
\end{proof}

In order to exploit \cref{Aubin-Lions}, we need the following mass preservation property for weak solutions of the system~\eqref{non-local continuity}.

\begin{lem}[Mass preservation]
\label{mass}
Let $V$ and $\eta$ be as in~\eqref{V def} and~\eqref{eta def}, respectively.
If $\rho \in C([0,T];\MM^+(\R^d)^k\ws)$ is a weak solution of the system~\eqref{non-local continuity} starting from the initial datum $\overline{\rho} \in \MM^+(\R^d)^k$, then
\begin{equation}
\norm{\rho^i(t,\cdot)}_{\MM} = \norm{\bar{\rho}^i}_{\MM} 
\quad
\label{mass is preserved}
\end{equation}
for $t\in[0,T]$ and $ i=1,\dots,k$.
\end{lem}

\begin{proof}
Let $i \in\set*{ 1, \dots, k}$ be fixed.
By applying~\eqref{weak formulation} to the test function $\phi(t,x)=\alpha(t)\,\beta(x)$, $(t,x)\in[0,T]\times\R^d$, where $\alpha\in C^\infty_c([0,T))$ and $\beta\in C^\infty_c(\R^d)$, we get
\begin{equation*}
\int_0^T\int_{\R^d} \left(
\alpha'\beta
+
\alpha\,V^i(t,x,\rho*\eta^i)\cdot\nabla\beta
\right)\di\rho^i(t,\cdot)\di t+
\alpha(0)\int_{\R^d}\beta(x)\di\bar\rho^i=0.
\end{equation*}
Since $\alpha\in C^\infty_c([0,T))$ is arbitrary and $\rho \in C([0,T];\MM^+(\R^d)^k\ws)$, we infer that 
\begin{equation}
t\mapsto
\int_{\R^d}\beta\di\rho^i(t,\cdot)
\in AC^{1,1}([0,T];\R)
\label{beta def ac}
\end{equation}
with
\begin{equation}
\int_{\R^d}\beta\di\rho^i(t,\cdot)
=
\int_{\R^d}\beta\di\bar\rho^i
+
\int_0^t
\int_{\R^d}V^i(s,\cdot,\rho*\eta^i)\cdot\nabla\beta
\di\rho^i(s,\cdot)
\di s
\label{beta sobolev}
\end{equation}
for all $t\in[0,T]$.
Now let $t \in [0,T]$ be fixed. 
We let $(\beta_R)_{R>0}\subset C^\infty_c(\R^d)$ be such that 
\begin{equation*}
\beta_R\ge0,
\quad
\mathrm{supp}\,\beta_R\subset B_{2R}, \quad
\beta_R=1\ \text{on}\ B_R,
\quad
\|\nabla\beta_R\|_{L^\infty}
\le\frac2R
\end{equation*}
for all $R>0$. 
By the Monotone Convergence Theorem, we infer that 
\begin{equation*}
\lim_{R\to+\infty}
\int_{\R^d}\beta_R\di\rho^i(t,\cdot)
=
\|\rho^i(t,\cdot)\|_{\MM}
\end{equation*}
as well as
\begin{equation*}
\lim_{R\to+\infty}
\int_{\R^d}\beta_R\di\bar\rho^i
=
\|\bar\rho^i\|_{\MM}.
\end{equation*}
Since 
\begin{equation*}
\abs*{\,\int_0^t
\int_{\R^d}V^i(s,\cdot,\rho*\eta^i)\cdot\nabla\beta_R
\di\rho^i(s,\cdot)
\di s \,}
\le 
\frac2R
\,
\norm{\rho^i}_{L^\infty(\MM)}
\,
\norm{V^i}_{L^\infty(C)}
\end{equation*}
for all $R>0$, we get~\eqref{mass is preserved} by applying~\eqref{beta sobolev} to $\beta_R$ and passing to the limit as $R\to+\infty$.
\end{proof}

We are ready to prove our existence result.

\begin{proof}[Proof of \cref{existence}] 
Let $(\ell_\e)_{\e>0}\subset C^{\infty}_c (\R^{d+k})$ and $(\jmath_\e)_{\e>0}\in C^\infty_c (\R^d)$ be two families of standard non-negative mollifiers and set
\begin{equation*}
V_\e^{i,j}(t,\cdot) = V^{i,j}(t,\cdot) * \ell_\e , 
\quad
\eta_\e^{i,j} = \eta^{i,j}(t,\cdot) * \jmath_\e,
\end{equation*}
where in both cases the (component-wise) convolution occur in the spatial variables only. 
Since $V_\e$ and $\eta_\e$ clearly satisfy the Lipschitz property~\eqref{lip ass} for each $\e>0$, by~\cite{CL13}*{Th.~1.1} there exists a weak solution 
\begin{equation*}
\rho_\e \in C([0,T],\MM^+(\R^d)^k\ws)
\end{equation*} 
of the system~\eqref{non-local continuity} starting from the initial datum $\bar{\rho} \in \MM^+(\R^d)^k$, so that 
\begin{equation}
\int_0^T\int_{\R^d} \left(
\partial_t \phi
+
V^i_\e(t,\cdot,\rho_\e*\eta^i_\e)\cdot\nabla\phi
\right)\di\rho^i_\e(t,\cdot)\di t+\int_{\R^d}\phi(0,\cdot)\di\bar\rho^i=0
\label{approximated weak}
\end{equation} 
for each $i=1, \dots,k$ and $\e>0$ and $\phi \in C^{\infty}_c([0,T)\times\R^d)$.

Now let $i \in \set{1,\dots, k}$ be fixed. 
We claim that (any sequence in) the family $(\rho^i_\e)_{\e>0}$ satisfies the assumptions~\eqref{equibound norm} and~\eqref{equicont} of \cref{Aubin-Lions}. 
Indeed, from \cref{mass} we get 
\begin{equation}
\|\rho_\e^i(t,\cdot)\|_{\MM}
=
\|\bar\rho^i\|_{\MM}
\label{mass preserv eps}
\end{equation} 
for all $t\in[0,T]$ and $\e>0$, from which~\eqref{equibound norm}  immediately follows.
To prove~\eqref{equicont}, we simply argue as in the proof of \cref{mass}.
Recalling~\eqref{beta def ac} and~\eqref{beta sobolev}, we easily recognize that the time derivative of the function
\begin{equation}
F_\e[\beta](t)
=
\int_{\R^d}\beta(\cdot)\di\rho^i_\e(t,\cdot),
\quad
t\in[0,T],
\label{beta eps def}
\end{equation}
is bounded by
\begin{equation*}
\begin{split}
\abs*{\,
\int_{\R^d}V^i_\e(t,x,\rho_\e*\eta^i_\e)\cdot\nabla\beta
\di\rho^i_\e(t,x)\,}
&\le 
\norm{V^i}_{L^\infty(C)}
\,
\|\nabla\beta\|_{L^\infty}
\,
\norm{\bar\rho^i}_{\MM}
\end{split}
\end{equation*}
for a.e.\ $t\in[0,T]$ and for each $\e>0$.
In particular, the family $(F_\e[\beta])_{\e>0}$ in~\eqref{beta eps def} is equi-Lipschitz and thus satisfies~\eqref{equicont}.
Therefore, by \cref{Aubin-Lions}, we find a sequence $(\rho_{\e_n})_{n\in\N} \subset C([0,T]; \MM^+(\R^d)^k\ws)$ and $\rho\in C([0,T]; \MM^+(\R^d)^k\ws)$ such that 
\begin{equation}
\lim_{n\to+\infty}
\sup_{t\in[0,T]}\,
\abs*{\,\int_{\R^d}\beta\di\rho_{\e_n}(t,\cdot)-\int_{\R^d}\beta\di\rho(t,\cdot)\,}=0
\label{conv eps_n}
\end{equation}
for all $\beta \in C_0(\R^d)$. 

To conclude, we just need to prove that $\rho$ is a weak solution of~\eqref{non-local continuity} starting from the initial datum $\bar{\rho}$.
We do so by passing to the limit in~\eqref{approximated weak} along $(\e_n)_{n\in\N}$ as $n\to+\infty$ for each given  $\phi \in C^\infty_c([0, +\infty) \times \R^d)$.
Indeed, on the one side, since 
\begin{equation*}
\lim_{n\to+\infty}
\int_{\R^d}
\partial_t \phi
\di\rho^i_{\e_n}(t,\cdot)
=
\int_{\R^d}
\partial_t \phi
\di\rho^i(t,\cdot)
\end{equation*}
because of~\eqref{conv eps_n} and 
\begin{equation*}
\abs*{
\,
\int_{\R^d}
\partial_t \phi
\di\rho^i_{\e_n}(t,\cdot)
\,
} 
\le
\norm{\partial_t\phi}_{L^\infty}
\,
\norm{\bar\rho^i}_{\MM}
\end{equation*}
because of~\eqref{mass preserv eps}, for all $t\in[0,T]$,
by the Dominated Convergence Theorem we infer that 
\begin{equation}
\lim_{n\to+\infty}
\int_0^T\int_{\R^d}
\partial_t\phi
\di\rho^i_{\e_n}(t,\cdot)
\di t
=
\int_0^T\int_{\R^d}
\partial_t\phi
\di\rho^i(t,\cdot)
\di t.
\label{first piece}
\end{equation} 
On the other side, since $\eta^i(t,\cdot)\in C_0(\R^d)$ in virtue of the assumption~\eqref{ass eta}, we have that $\eta^i_{\e_n}(t,\cdot)\to\eta^i(t,\cdot)$ in $C_0(\R^d)$ as $n\to+\infty$, so that
\begin{equation}
\begin{split}
\lim_{n\to+\infty}
\big(
\rho_{\e_n}(t,\cdot)*\eta^i_{\e_n}(t,\cdot)
\big)(x)
&=
\lim_{n\to+\infty}
\int_{\R^d}
\eta^i_{\e_n}(t,x-y)
\di\rho_{\e_n}(t,y)
\\
&=
\int_{\R^d}
\eta^i(t,x-y)
\di\rho(t,y)
=
\big(
\rho(t,\cdot)*\eta^i(t,\cdot)
\big)(x)
\label{pointw rho eta n}
\end{split}
\end{equation}
for each $x\in\R^d$ and all $t\in[0,T]$ as a weak-strong convergent pair, due to~\eqref{conv eps_n}.
Moreover, again in virtue of~ \eqref{mass preserv eps} and~\eqref{ass eta}, we can estimate
\begin{equation*}
\norm{
\rho_{\e_n}(t,\cdot)*\eta^i_{\e_n}(t,\cdot)
}
\le 
\norm{\overline{\rho}^i}_{\MM}
\,
\norm{\eta^i}_{L^\infty(C)}
\end{equation*}
and 
\begin{equation*}
\begin{split}
\big|
\big(
\rho_{\e_n}(t,\cdot)*\eta^i_{\e_n}&(t,\cdot)
\big)(x)
-
\big(
\rho_{\e_n}(t,\cdot)*\eta^i_{\e_n}(t,\cdot)
\big)(y)
\big|
\\
&\le 
\int_{\R^d}
\abs*{\eta^i_{\e_n}(t,x-\cdot)-\eta^i_{\e_n}(t,y-\cdot)}
\di\rho_{\e_n}(t,\cdot)
\le 
\omega_\eta(|x-y|)
\,
\norm{\overline{\rho}^i}_{\MM}
\end{split}
\end{equation*}
for all $n\in\N$ and $t\in[0,T]$.
By Arzelà--Ascoli's Theorem, we thus get that the pointwise convergence in~\eqref{pointw rho eta n} must be uniform on compact sets in $\R^d$, uniformly in $t\in[0,T]$. 
An analogous argument relying on the assumption~\eqref{ass V} proves that also $V^i_{\e_n}(t,\cdot)\to V^i(t,\cdot)$ as $n\to+\infty$ uniformly on compact sets in $\R^d$, uniformly in $t\in[0,T]$.
Again by~\eqref{conv eps_n}, by weak-strong convergence and by the Dominated Convergence Theorem, we hence get 
\begin{equation}
\lim_{n\to+\infty}
\int_0^T\int_{\R^d} 
V^i_{\e_n}(t,\cdot,\rho_{\e_n}*\eta^i_{\e_n})\cdot\nabla\phi
\di\rho^i_{\e_n}(t,\cdot)\di t
=
\int_0^T\int_{\R^d} 
V^i(t,\cdot,\rho*\eta^i)\cdot\nabla\phi
\di\rho^i(t,\cdot)\di t. 
\label{second piece}
\end{equation}
Thus, the conclusion follows by combining~\eqref{first piece} with~\eqref{second piece}. 
\end{proof}

\subsection{Lagrangian stability}

We deal with the  Lagrangian stability of weak solutions. We begin with the proof of \cref{vector field omega}.

\begin{proof}[Proof of \cref{vector field omega}]
Let $t\in[0,T]$ be fixed.
Given $x,y\in\R^d$ and $i\in\set*{1,\dots,k}$, in virtue of assumption~\eqref{ass eta} and of \cref{mass}, we can estimate
\begin{equation*}
\begin{split}
|\rho*\eta^i(t,x)-\rho*\eta^i(t,y)|
&\le 
\sum_{j=1}^k
\int_{\R^d}|\eta^{i,j}(t,x-z)-\eta^{i,j}(t,y-z)|\di\rho^j(t,z)
\\
&\le 
\sum_{j=1}^k
\int_{\R^d}\omega_\eta(|x-y|)\di\rho^k(t,z)
=
\|\rho(t,\cdot)\|_{\MM}\,\omega_\eta(|x-y|)
\\
&= 
\|\bar\rho\|_{\MM}\,\omega_\eta(|x-y|).
\end{split}
\end{equation*}
Thus, thanks to assumption~\eqref{ass V}, we get that 
\begin{equation*}
\begin{split}
\abs*{
V^i\big(t,x,\rho*\eta^i(t,x)\big)
-
V^i\big(t,y,\rho*\eta^i(t,y)\big)
}
&\le 
\omega_V\big(|x-y|+|\rho*\eta^i(t,x)-\rho*\eta^i(t,y)
|\big)
\\
&\le
\omega_V\big(
|x-y|+\|\bar\rho\|_{\MM}\,\omega_\eta(|x-y|)
\big)
\end{split}
\end{equation*}
and the conclusion immediately follows.
\end{proof}

We conclude our paper with the proof of \cref{stability}.

\begin{proof}[Proof of \cref{stability}] 
Let $V,U$, $\eta,\nu$, $\bar{\rho}, \bar{\sigma}$, $X,Y$ and $\rho,\sigma$ be as in the statement.
Fix $\zeta\in C(\R^d)$ with $\zeta\ge0$ and $\int_{\R^d} \zeta(x) \, \di x = 1$. Letting $\mu\in\MM^+(\R^d)$ be defined by
$\mu=|\bar\rho|+|\bar\sigma|+\zeta\,\mathscr L^d$, we consider the quantity
\begin{equation*}
Q_{\zeta}(t)
=
\sum_{i=1}^k
\mint{-}_{\R^d}
|X^i(t,\cdot)-Y^i(t,\cdot)|\di\mu
\end{equation*}
for all $t\in[0,T]$.
Note that $t\mapsto Q_\zeta(t)$ is well defined and Lipschitz, with $Q_\zeta(0)=0$ and  
\begin{equation*}
|Q_\zeta(s)-Q_\zeta(t)|
\le 
k\,(\|U\|_{L^\infty(C)}+\|V\|_{L^\infty(C)})
\,
|s-t|
\end{equation*}
for all $s,t\in[0,T]$.
Therefore, for a.e.\ $t\in[0,T]$, we can write
\begin{equation*}
\begin{split}
Q_\zeta'(t) 
& 
\leq 
\sum_{i=1}^k
\mint{-}_{\R^d} |V^i(t,X^i, \rho *\eta^i(t,X^i))
- U^i(t,Y^i, \sigma*\nu^i (t,Y^i))|\di\mu
\\
&\le
\sum_{i=1}^k\  
(1)_i+(2)_i+(3)_i+(4)_i,
\end{split}
\end{equation*}
where (dropping the variables of $X$ and $Y$ for notational convenience)
\begin{equation*}
\begin{split}
(1)_i
&=
\mint{-}_{\R^d} 
|V^i(t,X^i, \rho *\eta^i(t,X^i))
- V^i(t,Y^i,\rho *\eta^i (t,Y^i))|
\di\mu,
\\
(2)_i
&=
\mint{-}_{\R^d} 
|V^i(t,Y^i,\rho *\eta^i (t,Y^i))-V^i(t,Y^i,\sigma*\eta^i (t,Y^i))|
\di\mu,
\\
(3)_i
&=
\mint{-}_{\R^d} |V^i(t,Y^i,\sigma*\eta^i (t,Y^i))-V^i(t,Y^i,\sigma*\nu^i (t,Y^i))|
\di\mu,
\\
(4)_i
&=
\mint{-}_{\R^d} |V^i(t,Y^i,\sigma*\nu^i (t,Y^i))-U^i(t,Y^i,\sigma*\nu^i (t,Y^i))|
\di\mu.
\end{split}
\end{equation*}
We now estimate each term separately at a given $t\in[0,T]$.
By \cref{vector field omega} and Jensen's inequality, we can easily estimate the first term as
\begin{equation*}
\begin{split}
(1)_i
&\le 
\mint{-}_{\R^d} 
\omega_V\big(|X^i-Y^i|+\|\bar\rho\|_{\MM}\,\omega_\eta(|X^i-Y^i|)\big)
\di\mu
\\
&\le  
\omega_V\left(
\mint{-}_{\R^d}
|X^i-Y^i|\di\mu
+
\|\bar\rho\|_{\MM}\,\omega_\eta\left(\mint{-}_{\R^d}
|X^i-Y^i|\di\mu\right)
\right)
\\
&\le 
\omega_V\big(Q_\zeta(t)+\|\mu\|_{\MM}\,\omega_\eta(Q_\zeta(t))\big).
\end{split}
\end{equation*}
Concerning the second term, since 
\begin{equation*}
\begin{split}
|(\rho-\sigma)&*\eta^i(t,x)|
=
\abs*{
\int_{\R^d}\eta^i(t,x-y)\di\,(X_\#\bar\rho(y)-Y_\#\bar\sigma(y))
}
\\
&\le
\sum_{j=1}^k
\int_{\R^d}|\eta^{i,j}(t,x-X^j)-\eta^{i,j}(t,x-Y^j)| 
\di\bar\rho^j
+
\int_{\R^d}|\eta^{i,j}(t,x-Y^j)| 
\di|\bar\rho^j-\bar\sigma^j|
\\
&\le 
\sum_{j=1}^k
\int_{\R^d}\omega_\eta(|X^j-Y^j|) 
\di\bar\rho^j
+
\|\eta\|_{L^\infty(C)} 
\|\bar\rho^j-\bar\sigma^j\|_{\MM}
\\
&\le
\int_{\R^d}\omega_\eta\left(\sum_{j=1}^k|X^j-Y^j|\right) 
\di|\bar\rho|
+
\|\eta\|_{L^\infty(C)} 
\|\bar\rho-\bar\sigma\|_{\MM}
\end{split}
\end{equation*}
for all $x\in\R^d$, again by Jensen's inequality we get 
\begin{equation*}
\begin{split}
(2)_i
&\le 
\mint{-}_{\R^d} 
\omega_V\big(|(\rho-\sigma)*\eta^i(t,Y^i)|\big)
\di\mu
\\
&\le
\omega_V\left(
\int_{\R^d}\omega_\eta\left(\sum_{i=1}^k|X^i-Y^i|\right)
\di|\bar\rho|
+
\|\eta\|_{L^\infty(C)} 
\|\bar\rho-\bar\sigma\|_{\MM} 
\right)
\\
&\le 
\omega_V\left(
\|\mu\|_{\MM}\,\omega_\eta(Q_\zeta(t)) 
+
\|\eta\|_{L^\infty(C)} 
\|\bar\rho-\bar\sigma\|_{\MM} 
\right)
\\
&\le 
\omega_V\big(Q_\zeta(t)+
\|\mu\|_{\MM}\,\omega_\eta(Q_\zeta(t))
\big) 
+
\omega_V\left(
\|\eta\|_{L^\infty(C)} 
\|\bar\rho-\bar\sigma\|_{\MM} 
\right).
\end{split}
\end{equation*}
The last two terms can be trivially estimated as
\begin{equation*}
\begin{split}
(3)_i
&\le 
\omega_V\big(\|\sigma\|_{L^\infty( \MM)}\,\|\eta-\nu\|_{L^\infty(C)}\big)
=
\omega_V\big(\|\bar\sigma\|_{\MM}\,\|\eta-\nu\|_{L^\infty(C)}\big)
\\
&\le 
\omega_V\big(\|\mu\|_{\MM}\,\|\eta-\nu\|_{L^\infty(C)}\big)
\end{split}
\end{equation*}
thanks to \cref{mass}, and 
\begin{equation*}
(4)_i
\le 
\|V-U\|_{L^\infty(C)}.
\end{equation*}
Putting everything altogether, we conclude that 
\begin{equation*}
Q_\zeta'(t)
\lesssim 
\omega_V\big(Q_\zeta(t)+\lambda \,\omega_\eta(Q_\zeta(t))
\big) 
+
M,
\end{equation*}
where 
$\lambda = \|\overline{\rho}\|_{\MM} + \| \overline{\sigma}\|_{\MM} + 1$
and 
\begin{equation*}
\begin{split}
M
&=
\omega_V\left(
\|\eta\|_{L^\infty(C)} 
\|\bar\rho-\bar\sigma\|_{\MM} 
\right)
+
\omega_V\big(\lambda\,\|\eta-\nu\|_{L^\infty(C))}\big)
+
\|V-U\|_{L^\infty(C)}.
\end{split}
\end{equation*}
At this point, we just need to recall the Osgood condition assumed in~\eqref{ass osgood} and the integral function in~\eqref{bls}.
Indeed, by the classical Bihari--LaSalle inequality (see \cite{pachpatte}*{Th.~2.3.1} for instance), we find a modulus of continuity $\Omega\colon[0,+\infty)\to[0,+\infty)$, only depending on
\begin{equation*}
T,\
\|\bar\rho\|_{\MM},\
\|\bar\sigma\|_{\MM},\
\|\eta\|_{L^\infty(C)},\
\|\nu\|_{L^\infty(C)},\
\omega_V,\
\omega_\eta,
\end{equation*}
such that 
\begin{equation}
\sup_{t\in[0,T]} Q_\zeta(t) 
\le 
\Omega\left(
\|\bar\rho-\bar\sigma\|_{\MM}
+
\|V-U\|_{L^\infty(C)}
+
\|\nu-\eta\|_{L^\infty(C)}
\right).
\label{stability estimate 1}
\end{equation}
We remark that $\Omega$ is independent of $\zeta$, as long as we choose $\zeta\ge0$ and $\| \zeta\|_{L^1} =1$. To conclude, we choose a family $(\zeta_{x_0,\e})_{\e>0}$ of standard mollifiers around $x_0 \in \R^d$. Since the flows $X(t,\cdot), Y(t,\cdot)$ are continuous maps, we deduce that 
\begin{equation}
 \lim_{\e \to 0^+} Q_{\zeta_{x_0, \e}} (t) = \abs{X(t,x_0)- Y(t, x_0)}. 
\label{stability estimate 2}
\end{equation}
Thus,~\eqref{stability estimate} follows from~\eqref{stability estimate 1} and~\eqref{stability estimate 2} and the proof is complete. 
\end{proof}



\end{document}